\documentclass
[reqno,12pt]{amsart}
\usepackage{url}
\usepackage{amssymb,amsthm,amsmath,amsfonts,amstext,mathrsfs}
\usepackage{latexsym}
\usepackage[all]{xy}

\newtheorem{prop}{Proposition}
\newtheorem{thm}{Theorem}
\newtheorem{coll}{Colloary}
\newtheorem{lem}{Lemma}
\newtheorem{defn}{Definition}

\newcommand{\sE}{{\mathcal E}}
\newcommand{\sF}{{\mathcal F}}

\newcommand{\sK}{{\mathcal K}}
\newcommand{\sL}{{\mathcal L}}
\newcommand{\sO}{{\mathcal O}}

\newcommand{\sQ}{{\mathcal Q}}

\input xy
\xyoption{all}

\title
{Remarks on Xiao's approach of Slope inequalities}
\thanks{Hao Sun is supported by the National Natural Science Foundation of China (No. 11301201); Xiaotao Sun is supported by the National Natural Science Foundation of China (No.11321101);  Mingshuo Zhou is supported by the
National Natural Science Foundation of China (No. 11501154) and Natural Science Foundation of Zhejiang Provincial (No. LQ16A010005).}
\author{Hao Sun, Xiaotao Sun and Mingshuo Zhou}
\begin{document}


\begin{abstract}
 We prove the slope inequality for a relative minimal surface fibration in positive characteristic via Xiao's approach. We also prove a better low bound for the slope of non-hyperelliptic fibrations.

\end{abstract}

\maketitle

\section{Introduction}
Let $S$ be a smooth projective surface over an algebraically closed field $\mathbf{k}$ of characteristic $p\ge 0$ and $f:S\rightarrow B$ be a fibration with smooth general  fiber $F$ of genus $g$ over a smooth projective curve $B$. Let $\omega_{S/B}:=\omega_S\otimes f^{\ast}\omega_B^{\vee}$ be the relative canonical sheaf of $f$, and $K_{S/B}:=K_S-f^{\ast}K_B$ be the relative canonical divisor. We say that $f$ is relatively minimal if $S$ contains no $(-1)$-curve in fibers. The following basic relative invariants are well known:
$$\aligned & K_{S/B}^2= (K_S-f^{\ast}K_B)^2= K_S^2-8(g-1)(b-1),\\
& \chi_f= \mathrm{deg}f_{\ast}\omega_{S/B}=\chi(\sO_S)-(g-1)(b-1).\endaligned$$
\noindent When $f$ is relatively minimal and $F$ is smooth, then $K_{S/B}$ is a nef divisor (see \cite{X4}).
Under this assumption, the relative invariants satisfy the following remarkable so-called slope inequality.
\begin{thm}\label{thm1} If $f$ is relatively minimal, and the general fiber $F$ is smooth, then
\begin{align*} K_{S/B}^2\geq \frac{4(g-1)}{g}\chi_f.
\tag{1.1} \end{align*}
\end{thm}
When $char(\mathbf{k})=0$, this inequality was proved by Xiao (see \cite{XG}). For the case of semi-stable fibration, it was proved independently by Cornalba-Harris (see \cite{C.H}). When $char(\mathbf{k})=p>0$, there exist a few approach to prove this inequality (see \cite{M}, \cite{Y.Z}, ect). Some of them require the condition of
semi-stable fibration.

In this note, we explain why Xiao's approach still works in the case of $char(\mathbf{k})=p>0$. Indeed, Xiao's approach is to study the Harder-Narasimhan filtration
$$0=E_0\subset E_1\subset \cdots \subset E_n=E=f_{\ast}\omega_{S/B}$$
and give lower bound of $K_{S/B}^2$ in term of slop $\mu_i=\mu(E_i/E_{i-1})$. Here one of the key points
is that semi-stability of $E_i/E_{i-1}$ will imply nefness of $\mathbb{Q}$-divisors $\mathcal{O}_{\mathbb{P}(E_i)}(1)-\mu_i\Gamma_i$ where $\Gamma_i$ is a fiber of
$\mathbb{P}(E_i)\to B$. This is the only place one needs $char(\mathbf{k})=0$.

Our observation is that by a result of A. Langer there is an integer $k_0$ such that, when $k\ge k_0$, the Harder-Narasimhan filtration
$$0=E_0\subset E_1\subset \cdots \subset E_n=E=F^{k*}f_{\ast}\omega_{S/B}$$
of $F^{k*}f_{\ast}\omega_{S/B}$ has strongly semi-stable $E_i/E_{i-1}$ ($1\le i\le n$) and that strongly semi-stability of $E_i/E_{i-1}$ implies
nefness of $\mathcal{O}_{\mathbb{P}(E_i)}(1)-\mu_i\Gamma_i$. When $f: S\to B$ is a semi-stable fibration, any Frobenius base change $F^k:B\to B$ induces
fibration $\tilde{f}: \widetilde{S}\to B$ such that
$$F^{k*}f_{\ast}\omega_{S/B}=\tilde{f}_*\omega_{\widetilde{S}/B},\quad \frac{K_{S/B}^2}{\mathrm{deg}f_{\ast}\omega_{S/B}}=\frac{K_{\widetilde{S}/B}^2}{\mathrm{deg}\tilde{f}_{\ast}\omega_{\widetilde{S}/B}}.$$
Thus for semi-stable fibration $f:S\to B$ we can assume (without loss of generality) that all $E_i/E_{i-1}$ appearing in Harder-Narasimhan filtration of
$E=f_*\omega_{S/B}$ are strongly semi-stable. Then Xiao's approach works for $char(\mathbf{k})=p>0$ without any modification. We will show in this note that
a slightly modification of Xiao's approach works for any fibration $f:S\to B$. In fact, we will prove the following more general result holds for
$char(\mathbf{k})=p\ge 0$.

\begin{thm} Let $D$ be a relative nef divisor on $f:S\to B$ such that $D|_{F}$ is generated by global sections on a general smooth fiber $F$ of
$f:S\to B$. Assume that $D|_{F}$ is a special divisor on $F$ and $$A=2h^0(D|_{F})-D\cdot F-1>0.$$ Then
$$D^2\ge \frac{2D\cdot F}{h^0(D|_{F})}{\rm deg}(f_*\sO_S(D)).$$
\end{thm}

Xiao also constructed examples (cf.\cite[Example~2]{XG}) of hyperelliptic fiberation $f:S\to B$ such that
$$K^2_{S/B}= \frac{4g-4}{g}{\rm deg}(f_*\omega_{S/B})$$
and conjectured (cf. \cite[Conjecture~1]{XG}) that the inequality must be strict for non-hyperelliptic fibrations, i.e., the general
fiber $F$ of $f$ is a non-hyperelliptic curve, which was proved by Konno \cite[Proposition~2.6]{K}. Lu and Zuo \cite{L.Z} obtained a sharp
slope inequality for non-hyperelliptic fibrations, which was generalized to
$char(\mathbf{k})=p>0$ in \cite{L.S} for a
non-hyperelliptic semi-stable fibration.

Here we also remark that our previous observation can be used to prove the following theorem in any
characteristic easily.

\begin{thm}\label{thm3}
Assume that $f:S\rightarrow B$ is a relatively minimal
non-hyperelliptic surface fibration over an algebraically closed
field of any characteristic, and the general fiber of $f$ is smooth.
Then
\begin{align*} K_{S/B}^2\geq \mathrm{min} \{\frac{9(g-1)}{2(g+1)},4\}{\rm deg}f_*\omega_{S/B}. \tag{1.3}\end{align*}
\end{thm}

Our article is organized as follows. In Section 2, we give a
generalization of Xiao's approach, and show that a slightly
modification of Xiao's approach works in any characteristic. In
Section 3, we prove Theorem \ref{thm3} via the modification of
Xiao's approach and the modified second multiplication map
$F^{k*}S^2f_*\omega_{S/B} \to F^{k*}f_*(\omega_{S/B}^{\otimes 2})$.

\section{Xiao's approach and its generalization}
We start from an elementary (but important) lemma due to Xiao.

\begin{lem} (\cite[Lemma~2]{XG})\label{lem1.1} Let $f: S\rightarrow B$ be a relatively minimal fibration, with a general fiber $F$. Let $D$ be a divisor on $S$, and suppose that there are a sequence of effective divisors $$Z_1\geq Z_2\geq \cdots \geq Z_n\geq Z_{n+1}=0$$
and a sequence of rational numbers $$\mu_1>\mu_2\cdots >\mu_n,\quad \mu_{n+1}=0$$
such that for every $i$, $N_i=D-Z_i-\mu_i F$ is a nef $\mathbb{Q}$-divisor. Then $$D^2\geq \sum_{i=1}^n(d_i+d_{i+1})(\mu_i-\mu_{i+1}),$$
where $d_i=N_i\cdot F$.
\end{lem}

\begin{proof} Since $N_{i+1}=N_i+(\mu_i-\mu_{i+1})F+(Z_i-Z_{i+1})$,  we have
$$\aligned N^2_{i+1}&=N_{i+1}N_i+d_{i+1}(\mu_i-\mu_{i+1})+N_{i+1}(Z_i-Z_{i+1})\\
&=N_i^2+(d_i+d_{i+1})(\mu_i-\mu_{i+1})+(N_i+N_{i+1})(Z_i-Z_{i+1})\\&\ge
N_i^2+(d_i+d_{i+1})(\mu_i-\mu_{i+1}).\endaligned$$  Thus $N^2_{i+1}-N^2_i\ge (d_i+d_{i+1})(\mu_i-\mu_{i+1})$ and
$$D^2=N^2_{n+1}=N^2_1+\sum^n_{i=1}(N^2_{i+1}-N^2_i)\ge \sum_{i=1}^n(d_i+d_{i+1})(\mu_i-\mu_{i+1}).$$
\end{proof}

We need some well-known facts about vector bundles on curves. Let $B$ be a smooth projective curve over $\mathbf{k}$, for a vector bundle $E$ on $B$, the slope of $E$ is defined to be $$\mu(E)=\frac{\mathrm{deg}E}{\mathrm{rk}(E)}$$ where
$\mathrm{rk}(E)$, $\mathrm{deg}E$ denote the rank and degree of $E$ (respectively). Recall that $E$ is said to be semi-stable (resp., stable) if for any nontrivial subbundle $E'\subsetneq E$, we have $$\mu(E')\leq \mu(E) \ \ \ \  (\text{resp.,} <).$$

If $E$ is not semi-stable, one has the following well-known theorem

\begin{thm} (Harder-Narasimhan filtration) For any vector bundle $E$ on $B$, there is a unique filtration $$0:=E_0\subset E_1\subset \cdots \subset E_n=E$$
which is the so called Harder-Narasimhan filtration, such that

(1) each quotient $E_i/E_{i-1}$ is semi-stable for $1\leq i\leq n$,

(2) $\mu_1>\cdots >\mu_n$, where $\mu_i:=\mu(E_i/E_{i-1})$ for $1\leq i\leq n$.
\end{thm}

The rational numbers $\mu_{max}(E):=\mu_1$ and $\mu_{min}(E):=\mu_n$ are important invariants of $E$. Let $\pi:\mathbb{P}(E)\to B$ be projective bundle and $\pi^*E\to \mathcal{O}_E(1)\to 0$ be the tautological quotient line bundle. Then the following lemma (which was proved by Xiao in another formulation) relating semi-stability of $E$ with nefness of $\mathcal{O}_E(1)$ only holds when $char(\mathbf{k})=0$.

\begin{lem} (\cite[Theorem~3.1]{MY}, See also \cite[Lemma~3]{XG})\label{lem1.2} Let $\Gamma$ be a fiber of $\pi:\mathbb{P}(E)\to B$. Then
$$\mathcal{O}_E(1)-\mu_{min}(E)\Gamma$$
is a nef $\,\,\mathbb{Q}$-divisor. In particular, for each sub-bundle $E_i$ in Harder-Narasimhan filtration of $E$, the divisor
$$\mathcal{O}_{E_i}(1)-\mu_i\Gamma_i$$
is a nef $\,\,\mathbb{Q}$-divisor, where $\Gamma_i$ is a fiber of $\mathbb{P}(E_i)\to B$.
\end{lem}

\begin{thm}\label{thm1.3} Let $D$ be a relative nef divisor on $f:S\to B$ such that $D|_F$ is generated by global sections on a general smooth fiber $F$ of
$f:S\to B$. Assume that $D|_F$ is a special divisor on $F$ and $$A=2h^0(D|_F)-D\cdot F-1>0.$$ Then
$$D^2\ge \frac{2D\cdot F}{h^0(D|_F)}{\rm deg}(f_*\sO_S(D)).$$
\end{thm}

\begin{proof} For a divisor $D$ on $f:S\to B$, $E=f_*\mathcal{O}_S(D)$ is a vector bundle of rank $h^0(D|_F)$ where $F$ is a general smooth
fiber of $f:S\to B$. Let
$$0:=E_0\subset E_1\subset \cdots \subset E_n=E$$
be the Harder-Narasimhan filtration of $E$ with $r_i=\mathrm{rk}(E_i)$ and $$\mu_i=\mu(E_i/E_{i-1})=\mu_{min}(E_i)\quad (1\le i\le n).$$
Let $\mathcal{L}_i\subset \mathcal{O}_S(D)$ be the image of $f^*E_i$ under sheaf homomorphism
$$f^*E_i\hookrightarrow f^*E=f^*f_*\mathcal{O}_S(D)\to \mathcal{O}_S(D),$$
which is a torsion-free sheaf of rank $1$ and is locally free on an open set $U_i\subset S$ of codimension at least $2$. Thus there is a morphism (over $B$)
$$\phi_i: U_i\to \mathbb{P}(E_i)$$
such that $\phi_i^*\mathcal{O}_{E_i}(1)=\mathcal{L}_i|_{U_i}$, which implies that $c_1(\mathcal{L}_i)-\mu_iF$ is nef by Lemma \ref{lem1.2}. Let
$D=c_1(\mathcal{L}_i)+Z_i$ ($1\le i\le n$). Then we get a sequence of effective divisors $Z_1\geq Z_2\geq \cdots \geq Z_n\geq 0$
and a sequence of rational numbers $\mu_1>\mu_2\cdots >\mu_n$
such that  $$N_i=D-Z_i-\mu_i F\,\,\,\,(1\le i\le n)$$ are nef $\mathbb{Q}$-divisors.
Note $N_i|_F=c_1(\sL_i)|_F\hookrightarrow D|_F$, one has surjection $$H^1(N_i|_F)\twoheadrightarrow H^1(D|_F).$$
Thus $N_i|_F$ is special since $D|_F$ is special, and
$$d_i=N_i\cdot F\ge 2 h^0(\sL_i|_F)-2=2r_i-2, \,\,(i=1,\,...,\,n)$$
by Clifford theorem. Since $D|_F$ is generated by global sections, $Z_n$ is supported on fibers of $f:S\to B$ and
$d_n=D\cdot F:=d_{n+1}$.

When $n=1$, we have $D^2=N^2_1+(D+N_1)\cdot Z_1+2\mu_1D\cdot F$ and
$$D^2\ge 2\mu_1D\cdot F=\frac{2D\cdot F}{h^0(D|_F)}{\rm deg}(f_*\sO_S(D))$$
since $Z_1$ is supported on fibers of $f:S\to B$ and $D$ is a relative nef divisor.
When $n>1$, by the same reason,
$$D^2=N^2_n+2\mu_nD\cdot F+(N_n+D)\cdot Z_n\ge N^2_n+2\mu_nD\cdot F$$
and, by using Lemma \ref{lem1.1} to $N^2_n$, we have
$$\aligned D^2&\ge \sum_{i=1}^{n-1}(d_i+d_{i+1})(\mu_i-\mu_{i+1})+2\mu_n D\cdot F\\
&\ge \sum_{i=1}^{n-1}(2r_i+2r_{i+1}-4)(\mu_i-\mu_{i+1})+2\mu_n D\cdot F\\
&\ge\sum_{i=1}^{n-1}(4r_i-2)(\mu_i-\mu_{i+1})+2\mu_n D\cdot F\\&=
4\sum_{i=1}^nr_i(\mu_i-\mu_{i+1})-2\mu_1-(4h^0(D|_F)-2D\cdot F-2)\mu_n\\&=
4{\rm deg}(f_*\sO_S(D))-2\mu_1-2A\mu_n\endaligned $$
where we use the equality (which is easy to check) that
$$\mathrm{deg}(f_*\sO_S(D))=\sum_{i=1}^n r_i(\mu_i-\mu_{i+1}).$$
Again by $D^2=N^2_n+2\mu_nD\cdot F+(N_n+D)\cdot Z_n$,
apply Lemma \ref{lem1.1} to $(Z_1\ge Z_n\ge 0,\,\mu_1>\mu_n)$, we have
$$D^2\ge (d_1+D\cdot F)(\mu_1-\mu_n)+2D\cdot F\mu_n\ge D\cdot F(\mu_1+\mu_n).$$
By using above two inequalities and eliminating $\mu_1$, we have
$$(2+D\cdot F)D^2-4D\cdot F {\rm deg}(f_*\sO_S(D))\ge -2(A-1)D\cdot F\mu_n$$
By eliminating $\mu_n$ (which is possible since we assume $A>0$), we have
$$(2A+D\cdot F)D^2-4D\cdot F {\rm deg}(f_*\sO_S(D))\ge 2(A-1)D\cdot F\mu_1.$$
By adding above two inequalities and using definition of $A$, we have
$$4h^0(D|_F)D^2-8{\rm deg}(f_*\sO_S(D))\ge 2(A-1)D\cdot F(\mu_1-\mu_n)\ge 0$$
which is what we want.
\end{proof}

\begin{coll}(Xiao's inequality) Let $f:S\to B$ be a relatively minimal fibration of genius $g\ge 2$. Then
$$K^2_{S/B}\ge \frac{4g-4}{g}{\rm deg}(f_*\omega_{S/B}).$$
\end{coll}

\begin{proof} Take $D=K_{S/B}$ (the relative canonical divisor), which satisfies all the assumptions in Theorem \ref{thm1.3}
with $h^0(D|_F)=g$, $D\cdot F=2g-2$ and $\sO_S(D)=\omega_{S/B}$.
\end{proof}

The only obstruction to generalize Xiao's method in positive characteristic is Lemma \ref{lem1.2}, which is not true in positive characteristic since
Frobenius pull-back of a semi-stable bundle may not be semi-stable. However, the following notion of strongly semi-stability enjoy nice property that pull-back under a finite map preserves semi-stability.
\begin{defn} The bundle $E$ is called strongly semi-stable (resp., stable) if its pullback by $k$-th power $F^k$ is semi-stable (resp., stable) for any integer $k\geq 0$, where $F$ is the Frobenius morphism $B\rightarrow B$.
\end{defn}

\begin{lem}(\cite[Theorem~3.1]{LA})\label{lem1.3} For any bundle $E$ on $B$, there exists an integer $k_0$ such that all of quotients $E_i/E_{i-1}$ $(1\leq i \leq n)$
appear in the Harder-Narasimhan filtration $$0:=E_0\subset E_1\subset \cdots \subset E_n=F^{k*}E$$
are strongly semi-stable whenever $k\geq k_0$.
\end{lem}

\begin{lem}\label{lem1.4} For each sub-bundle $E_i$ in the Harder-Narasimhan filtration
$$0:=E_0\subset E_1\subset \cdots \subset E_n=F^{k*}E$$
of $F^{k*}E$ (when $k\ge k_0$), the divisor $\mathcal{O}_{E_i}(1)-\mu_i\Gamma_i$
is a nef $\,\,\mathbb{Q}$-divisor, where $\Gamma_i$ is a fiber of $\mathbb{P}(E_i)\to B$ and $\mu_i=\mu(E_i/E_{i-1})$.
\end{lem}
\begin{proof}
The proof is just a modification of \cite[Theorem~3.1]{MY} since pull-back of strongly semi-stable bundles under a finite morphism are still strongly semi-stable. One can see \cite[Theorem~3.1, Page~464]{MY} for more details.
\end{proof}

We now can prove, by the same arguments, that Theorem \ref{thm1.3} still holds in positive characteristic.

\begin{thm}\label{thm1.4} Let $D$ be a relative nef divisor on $f:S\to B$ such that $D|_{F}$ is generated by global sections on a general smooth fiber $F$ of
$f:S\to B$. Assume that $D|_{\Gamma}$ is a special divisor on $F$ and $$A=2h^0(D|_{F})-D\cdot F-1>0.$$ Then
$$D^2\ge \frac{2D\cdot F}{h^0(D|_{F})}{\rm deg}(f_*\sO_S(D)).$$
\end{thm}

\begin{proof} It is enough to prove the theorem when $f:S\to B$ is defined over a base field $\mathbf{k}$ of characteristic $p>0$.
Let $F_S: S\rightarrow S$ denote the Frobenius morphism over $\mathbf{k}$. Then  we have the following commutative diagram (for any integer $k\ge k_0$):$$
\xymatrix{
  S \ar[d]_{f} \ar[r]^{F^k_S} & S \ar[d]^{f} \\
  B \ar[r]^{F^k} & B   }.$$
For a divisor $D$ on $f:S\to B$, $E=f_*\mathcal{O}_S(D)$ is a vector bundle of rank $h^0(D|_F)$ where $F$ is a general smooth
fiber of $f:S\to B$. Let
$$0:=E_0\subset E_1\subset \cdots \subset E_n=F^{k*}E$$
be the Harder-Narasimhan filtration of $F^{k*}E$ with $r_i=\mathrm{rk}(E_i)$ and $$\mu_i=\mu(E_i/E_{i-1})=\mu_{min}(E_i)\quad (1\le i\le n)$$
where we choose $k\ge k_0$ such that all quotients $E_i/E_{i-1}$ appears in above filtration are strongly semi-stable.

Let $\mathcal{L}_i\subset F^{k*}_S\mathcal{O}_S(D)$ be the image of $f^*E_i$ under sheaf homomorphism
$$f^*E_i\hookrightarrow f^*F^{k*}E=F^{k*}_Sf^*f_*\mathcal{O}_S(D)\to F^{k*}_S\mathcal{O}_S(D)=\mathcal{O}_S(p^kD),$$
which is a torsion-free sheaf of rank $1$ and is locally free on an open set $U_i\subset S$ of codimension at least $2$. Thus there is a morphism (over $B$)
$$\phi_i: U_i\to \mathbb{P}(E_i)$$
such that $\phi_i^*\mathcal{O}_{E_i}(1)=\mathcal{L}_i|_{U_i}$, which implies that $c_1(\mathcal{L}_i)-\mu_iF$ is nef by Lemma \ref{lem1.4}. Let
$p^kD=c_1(\mathcal{L}_i)+Z_i$ ($1\le i\le n$). Then we get a sequence of effective divisors $Z_1\geq Z_2\geq \cdots \geq Z_n\geq 0$
and a sequence of rational numbers $\mu_1>\mu_2\cdots >\mu_n$
such that  $$N_i=p^kD-Z_i-\mu_i F\,\,\,\,(1\le i\le n)$$ are nef $\mathbb{Q}$-divisors. Let $d_i=N_i\cdot F={\rm deg}(\sL_i|_{F})$, then
$$d_n=p^kD\cdot F:=d_{n+1}$$
since $D|_F$ is generated by global sections and $Z_n$ is supported on fibers of $f:S\to B$. For $1\le i<n$, there are $r_i={\rm rk}(E_i)$ sections
$$\{s_1,\,\,...\,,\, \,s_{r_i}\}\in H^0(\mathcal{O}_S(D)|_{F})$$ such that $\sL_i|_{F}\subset\mathcal{O}_S(p^kD)|_{F}$ is generated by the global sections
$s^{p^k}_1,\,\,...\,,\,\,s_{r_i}^{p^k}$.

Since $\mathcal{O}_S(D)|_{F}$ is special, the sub-sheaf $L_i\subset\mathcal{O}_S(D)|_{F}$ generated by
$$\{s_1,\,\,...\,,\, \,s_{r_i}\}\in H^0(\mathcal{O}_S(D)|_{F})$$
is special. Thus ${\rm deg}(L_i)\ge 2r_i-2$ by Clifford theorem. Then we have
$$d_i=N_i\cdot F={\rm deg}(\sL_i|_{F})=p^k{\rm deg}(L_i)\ge p^k(2r_i-2)\,\,(1\le i\le n).$$

When $n=1$, which means that $E=f_*\sO_S(D)$ is strongly semi-stable, the same proof of Theorem \ref{thm1.3} implies
$$D^2\ge 2\mu_1D\cdot F=\frac{2D\cdot F}{h^0(D|_{F})}{\rm deg}(f_*\sO_S(D)).$$

When $n>1$, since $Z_n$ is supported on fibers of $f:S\to B$, we have
$$p^{2k}D^2=N_n^2+2\mu_n p^kD\cdot F+(N_n+p^kD)\cdot Z_n\ge N^2_n+2\mu_np^kD\cdot F.$$
By $d_i\ge p^k(2r_i-2)$ and using Lemma \ref{lem1.1} to $N^2_n$, we have
$$\aligned p^{2k}D^2&\ge \sum_{i=1}^{n-1}(d_i+d_{i+1})(\mu_i-\mu_{i+1})+2\mu_n p^kD\cdot F\\
&\ge \sum_{i=1}^{n-1}p^k(2r_i+2r_{i+1}-4)(\mu_i-\mu_{i+1})+2\mu_np^k D\cdot F\\
&\ge p^k\sum_{i=1}^{n-1}(4r_i-2)(\mu_i-\mu_{i+1})+2\mu_n p^kD\cdot F\\&=
4p^k\sum_{i=1}^nr_i(\mu_i-\mu_{i+1})-2p^k\mu_1-p^k(4h^0(D|_{F})-2D\cdot F-2)\mu_n\\&=
4p^k{\rm deg}(F^{k*}f_*\sO_S(D))-2p^k\mu_1-2p^kA\mu_n\endaligned $$
where we set $\mu_{n+1}=0$ and use the equality (which is easy to check)
$$\mathrm{deg}(F^{k*}f_*\sO_S(D))=\sum_{i=1}^n r_i(\mu_i-\mu_{i+1}).$$
By $p^{2k}D^2=N_n^2+2\mu_n p^kD\cdot F+(N_n+p^kD)\cdot Z_n$,
apply Lemma \ref{lem1.1} to $(Z_1\ge Z_n\ge 0,\,\mu_1>\mu_n)$, we have
$$p^{2k}D^2\ge (d_1+p^kD\cdot F)(\mu_1-\mu_n)+2p^kD\cdot F\mu_n\ge p^kD\cdot F(\mu_1+\mu_n).$$
Altogether, we have the following inequalities
\begin{align*} {p^kD^2\ge 4{\rm deg}(F^{k*}f_*\sO_S(D))-2\mu_1-2A\mu_n} \tag{2.1}\end{align*}
\begin{align*} {p^kD^2\ge D\cdot F(\mu_1+\mu_n)} \tag{2.2}\end{align*}
By using (2.1) and (2.2), eliminating $\mu_1$, we have
$$(2+D\cdot F)p^kD^2-4D\cdot F {\rm deg}(F^{k*}f_*\sO_S(D))\ge -2(A-1)D\cdot F\mu_n$$
By eliminating $\mu_n$ (which is possible since we assume $A>0$), we have
$$(2A+D\cdot F)p^kD^2-4D\cdot F {\rm deg}(F^{k*}f_*\sO_S(D))\ge 2(A-1)D\cdot F\mu_1.$$
By adding above two inequalities and using definition of $A$, we have
$$4h^0(D|_{F})p^kD^2-8{\rm deg}(F^{k*}f_*\sO_S(D))\ge 2(A-1)D\cdot F(\mu_1-\mu_n)\ge 0$$
which and ${\rm deg}(F^{k*}f_*\sO_S(D))=p^k{\rm deg}(f_*\sO_S(D))$ imply
$$D^2\ge \frac{2D\cdot F}{h^0(D|_{F})}{\rm deg}(f_*\sO_S(D)).$$
\end{proof}

\begin{coll} Let $f:S\to B$ be a relatively minimal fibration of genius $g\ge 2$ over an algebraically closed field of characteristic $p\ge 0$. Then
$$K^2_{S/B}\ge \frac{4g-4}{g}{\rm deg}(f_*\omega_{S/B}).$$
\end{coll}

\begin{proof} Take $D=K_{S/B}$ (the relative canonical divisor), which satisfies all the assumptions in Theorem \ref{thm1.4}
with $h^0(D|_{F})=g$, $D\cdot F =2g-2$ and $\sO_S(D)=\omega_{S/B}$.
\end{proof}

\section{Slopes of non-hyperelliptic fibrations}

Xiao has constructed examples (cf.\cite[Example~2]{XG}) of hyperelliptic fiberation $f:S\to B$ such that
$$K^2_{S/B}= \frac{4g-4}{g}{\rm deg}(f_*\omega_{S/B})$$
and has conjectured (cf. \cite[Conjecture~1]{XG}) that the inequality must be strict for non-hyperelliptic fibrations.

\begin{prop} Let $f:S\to B$ be a non-hyperelliptic fibration of genus $g\ge 3$, if $f_*\omega_{S/B}$ is strongly semi-stable, then
\begin{align*}{K^2_{S/B}\ge \frac{5g-6}{g}{\rm deg}(f_*\omega_{S/B}).} \tag{3.1}\end{align*}
\end{prop}

\begin{proof} By Max Noether's theorem, the second multiplication map
$$\varrho:S^2f_*\omega_{S/B} \to f_*(\omega_{S/B}^{\otimes2})$$
is generically surjective for non-hyperelliptic fibrations $f: S\to B$. Let
$$S^2f_*\omega_{S/B}\twoheadrightarrow \sF:=\varrho(S^2f_*\omega_{S/B})\subset f_*(\omega_{S/B}^{\otimes2}).$$
Then $\sF$ is a vector bundle of rank ${\rm rk}(f_*(\omega_{S/B}^{\otimes2}))=3g-3$, and
\begin{align*} {{\rm deg}(\sF)\le {\rm deg}(f_*(\omega_{S/B}^{\otimes2}))=K^2_{S/B}+{\rm deg}(f_*\omega_{S/B}).} \tag{3.2}\end{align*}
On the other hand, semi-stability of $S^2f_*\omega_{S/B}$ implies
\begin{align*}{{\rm deg}(\sF)\ge (3g-3)\mu(S^2f_*\omega_{S/B})=\frac{6g-6}{g}{\rm deg}(f_*\omega_{S/B}).} \tag{3.3}\end{align*}
Then (3.2) and (3.3) imply the required inequality (3.1).
\end{proof}

If $E=f_*\omega_{S/B}$ is not strongly semi-stable, let
\begin{align*}{0:=E_0\subset E_1\subset E_2\subset \cdots \subset E_{n-1}\subset \widetilde{E}=F^{k*}E}\tag{3.4}\end{align*}
be the Harder-Narasimhan filtration of $F^{k*}E$ with $r_i=\mathrm{rk}(E_i)$ and $$\mu_i=\mu(E_i/E_{i-1})=\mu_{min}(E_i)\quad (1\le i\le n)$$
where we choose $k\ge k_0$ such that all quotients $E_i/E_{i-1}$ appears in above filtration are strongly semi-stable.
The second multiplication map induces a multiplication map, which is still denoted by $\varrho$,
$$\varrho:S^2\widetilde{E}=F^{k*}S^2f_*\omega_{S/B} \to F^{k*}f_*(\omega_{S/B}^{\otimes 2}).$$
Let $\widetilde{\sF}=F^{k*}\sF=\varrho(S^2\widetilde{E})\subset F^{k*}f_*(\omega_{S/B}^{\otimes 2})$ be the image of $\varrho$, then
$$K^2_{S/B}\ge \frac{1}{p^k}{\rm deg}(\widetilde{\sF})-{\rm deg}(f_*\omega_{S/B}).$$
Thus the question is to find a good lower bound of ${\rm deg}(\widetilde{\sF})$, where
$$0\to\widetilde{\sK}:={\rm ker}(\varrho)\to S^2\widetilde{E}\xrightarrow{\varrho} \widetilde{\sF}\to 0.$$
Note that for any filtration
\begin{align*}{0:=\sF_0\subset \sF_1\subset \sF_2\subset \cdots \subset \sF_{n-1}\subset \sF_n:=\widetilde{\sF}}\tag{3.5}\end{align*}
of $\widetilde{\sF}$, ${\rm deg}(\sF_i/\sF_{i-1})\ge ({\rm rk}(\sF_i)-{\rm rk}(\sF_{i-1}))\mu_{min}(\sF_i)$. If $\mu_{min}(\sF_i)\ge a_i$,
\begin{align*}{{\rm deg}(\widetilde{\sF})\ge\sum^n_{i=1}{\rm rk}(\sF_i)(a_i-a_{i+1}).}\tag{3.6}\end{align*}
One of choices of the filtration (3.5) is induced by the Harder-Narasimhan filtration (3.4) of $\widetilde{E}=F^{k*}f_*\omega_{S/B}$
(similar with \cite{L.Z}):
$$\sF_i=\varrho(E_i\otimes E_i)\subset \widetilde{\sF}.$$
The following lemma implies that $\mu_{min}(\sF_i)\ge 2\mu_i$ for all $1\le i\le n$.

\begin{lem}
Let $\sE_1$ and $\sE_2$ be two bundles over a smooth projective curve
with all quotients in the Harder-Narasimhan of $\sE_1$ and $\sE_2$ are
strongly semi-stable. Then we have $$\mu_{min}(\sE_1\otimes \sE_2)=
\mu_{min}(\sE_1)+\mu_{min}(\sE_2).$$
\end{lem}
\begin{proof} It is clear that $\mu_{min}(\sE_1\otimes \sE_2)\le
\mu_{min}(\sE_1)+\mu_{min}(\sE_2)$ by \cite[Proposition~3.5 (3)]{S}.
Thus is enough we to show
$$\mu_{min}(\sE_1\otimes \sE_2)\ge
\mu_{min}(\sE_1)+\mu_{min}(\sE_2).$$ By Lemma \ref{lem1.3},
there is a $k_0$ such that for all $k\geq k_0$, all quotients in the
Harder-Narasimhan filtration of $F^{k\ast}(\sE_1\otimes \sE_2)$ are
strongly semi-stable.

Let $F^{k\ast}(\sE_1\otimes \sE_2)\twoheadrightarrow\sQ$ be the strongly semi-stable quotient with
$$\mu(\sQ)=\mu_{min}(F^{k\ast}(\sE_1\otimes \sE_2))\le p^k\mu_{min}(\sE_1\otimes\sE_2).$$
Applying \cite[Proposition~3.5(4)]{S} on the nontrivial morphism
$$F^{k\ast}\sE_1\rightarrow (F^{k\ast}\sE_2)^{\vee}\otimes \mathcal{Q},$$
we have $\mu_{max}((F^{k\ast}\sE_2)^{\vee}\otimes \mathcal{Q})\geq \mu_{min}(F^{k\ast}\sE_1)$
and $$\mu_{max}((F^{k\ast}\sE_2)^{\vee}\otimes \mathcal{Q})=\mu(\sQ)-\mu_{min}(F^{k\ast}\sE_2)$$
since all quotients $gr_i^{\mathrm{HN}}{(\sE_2)}$ and $\mathcal{Q}$
are strongly semi-stable. Then
$$\mu(\mathcal{Q})\geq\mu_{min}(F^{k*}\sE_1)+\mu_{min}(F^{k*}\sE_2)= p^k(\mu_{min}(\sE_1)+\mu_{min}(\sE_2)),$$
where the last equality holds since all $gr_i^{\mathrm{HN}}{(\sE_1)}$ and $gr_i^{\mathrm{HN}}{(\sE_2)}$ are strongly semi-stable, which implies that
$$\mu_{min}(\sE_1\otimes \sE_2)\ge
\mu_{min}(\sE_1)+\mu_{min}(\sE_2).$$
\end{proof}

A lemma of \cite{L.Z} provides the lower bound of ${\rm rk}(\sF_i)$. To state it, recall that in the proof of Theorem \ref{thm1.4}, each $E_i$
defines a morphism
$$\phi_{L_i}:F\to \mathbb{P}^{r_i-1}$$
on the general fiber $F$ of $f:S\to B$, where $L_i\subset \omega_F$ is generated by global sections
$\{s_1,\,\,...\,,\, \,s_{r_i}\}\subset
H^0(\sO_S(K_{S/B})|_{F})=H^0(\omega_F).$
\begin{defn} Let $\tau_i:C_i\to \phi_{L_i}(F)$ be the normalization of $\phi_{L_i}(F)$, $$g_i=g(C_i)$$ be the genius of $C_i$ and
$\psi_i:F\to C_i$ be the morphism such that
$$\phi_{L_i}=\tau_i\cdot\psi_i.$$
Let $c_i={\rm deg}(\phi_{L_i})={\rm deg}(\psi_i)$. Then $c_i|c_{i-1}$ for all $1\le i\le n$ and
$$r_1<r_2<\cdots <r_{n-1}<r_n=g,\quad g_1\le g_2\le\cdots\le g_{n-1}\le g_n=g.$$
\end{defn}

\begin{lem}(\cite[Lemma~2.6]{L.Z})
For each $1\le i\le n$, we have $$\mathrm{rk}(\mathcal{F}_i)\geq \begin{cases}
3r_i-3,    &\text{if} \ r_i\le g_i+1 ;\\
2r_i+g_i-1, & \text{if} \ r_i\geq g_i+2.
\end{cases}$$
In particular, if $\phi_{L_i}$ is a birational morphism, then
$$\mathrm{rk}(\mathcal{F}_i)\geq 3r_i-3.$$
\end{lem}

\begin{lem}
Let $d'_i$ be the degree of $\phi_{L_i}(F)\subset\mathbb{P}^{r_i-1}$,
$\ell=\mathrm{min}\{\,i\,\,| \,\,c_i=1\}$, $$I=\{1\le i\le\ell-1\,|\,\,r_i\ge g_i+2\,\}.$$
Then we have
$$p^kK^2_{S/B}\ge\sum_{i\in I}(3r_i+2g_i-2)(\mu_i-\mu_{i+1})+\sum_{i\notin I}(5r_i-6)(\mu_i-\mu_{i+1}),$$
$$p^kK^2_{S/B}\ge 2\sum_{i\in I}c_id'_i(\mu_i-\mu_{i+1})+\sum_{i\notin I}(4r_i-2)(\mu_i-\mu_{i+1})-2\mu_n.$$

\end{lem}
\begin{proof}
The first inequality is from (3.6) by taking $a_i=2\mu_i$ and using estimate of ${\rm rk}(\sF_i)$ in Lemma 6. The second inequality is
from $$p^{2k}K^2_{S/B}\ge \sum^{n-1}_{i=1}(d_i+d_{i+1})(\mu_i-\mu_{i+1})+p^k(4g-4)\mu_n$$
by using $d_i\le d_{i+1}$, $d_i=p^kc_id_i'$ and $c_id'_i\ge 2r_i-2$.
\end{proof}

\begin{prop} If ${\rm min}\{\,c_i\,|\,i\in I\}\ge 3$, then
$$K^2_{S/B}\ge \frac{9(g-1)}{2(g+1)}{\rm deg}f_*\omega_{S/B}.$$
\end{prop}

\begin{proof} When ${\rm min}\{\,c_i\,|\,i\in I\}\ge 3$, use $d'_i\ge r_i-1$ and Lemma 7,
$$p^kK^2_{S/B}\ge \sum_{i\in I}(3r_i-2)(\mu_i-\mu_{i+1})+\sum_{i\notin I}(5r_i-6)(\mu_i-\mu_{i+1})$$
$$p^kK^2_{S/B}\ge \sum_{i\in I}(6r_i-6)(\mu_i-\mu_{i+1})+\sum_{i\notin I}(4r_i-2)(\mu_i-\mu_{i+1})-2\mu_n.$$
Take the average of above two inequalities, we have
\begin{align*} {K^2_{S/B}\ge \frac{9}{2}{\rm deg}f_*\omega_{S/B}-\frac{4\mu_1+\mu_n}{p^k}.} \tag{3.7}\end{align*}
On the other hand, by Lemma 1, we have
$K^2_{S/B}\ge \frac{(2g-2)(\mu_1+\mu_n)}{p^k}$ and
$K^2_{S/B}\ge \frac{2g-2}{p^k}\mu_1$, which and (3.7) implies
the required inequality.
\end{proof}

\begin{prop} If ${\rm min}\{\,c_i\,|\,i\in I_1\}=2$ and $g_i\ge \frac{g-1}{4}$ for $i\in I$ with $c_i=2$.
Then we have
$$K_{S/B}^2>\frac{9(g-1)}{2(g+1)}\mathrm{deg}f_{\ast}\omega_{S/B}.$$
\end{prop}

\begin{proof} It is a matter to estimate $d'_i$. Since $\phi_{L_i}(F)\subset \mathbb{P}^{r_i-1}$ is an irreducible
non-degenerate curve of degree $d_i'$, we have in general $d'_i\ge r_i-1$ and more precisely the so called Castelnuovo's bound
$$d'_i-1\ge \frac{g_i}{m_i}+\frac{m_i+1}{2}(r_i-2)$$
where $m_i=[\frac{d'_i-1}{r_i-2}]$ is the positive integer defined by $d'_i-1=m_i(r_i-2)+\varepsilon_i$ with $0\le\varepsilon_i<1$
(see \cite[Chapter~III, 2]{ACGJ}).

Let $I_1=\{\,i\in I\,|\,c_i=2\}$. Then for any $i\in I_1$,
$d'_i\ge r_i-1+g_i$ by Castelnuovo's bound (since $r_i\ge g_i+2\ge 2$). On the other hand,
$$8g_i\ge2g-2\ge 2d'_i\ge 2r_i-2+2g_i$$
implies that $3g_i\ge r_i-1$, which implies that
$$(3r_i+2g_i-2) +2c_id_i'\ge 9r_i-8, \quad \forall\,\,i\in I_1,$$
thus $(3r_i+2g_i-2) +2c_id_i'\ge 9r_i-8$ for all $i\in I$. Then the required inequality follows the same
arguments in Proposition 2.
\end{proof}

\begin{prop}(\cite[Theorem~3.1, 3.2]{C.S}) If there is an $i\in I$ such that $c_i=2$ and $g_i<\frac{g-1}{4}$. Then
$$ K_{S/B}^2\geq \frac{4(g-1)}{g-g_i}\mathrm{deg}f_{\ast}\omega_{S/B}.$$
\end{prop}

\begin{proof} [Proof of Theorem 3] When $n=1$ (i.e. $f_*\omega_{S/B}$ strongly semistable), Theorem 3 is true by Proposition 1.
When $n>1$, Theorem 3 is a consequence of Proposition 2, Proposition 3 and Proposition 4
since we have $g_i\geq 1$ if $c_i=2$.
\end{proof}

\noindent\address{Hao Sun: Department of Mathematics, Shanghai Normal University, Shanghai 200234, P. R. of China.}\\
Email: hsun@shnu.edu.cn\\
\noindent\address{Xiaotao Sun: Institute of Mathematics and University of Chinese Academy of Sciences, P. R. of China.}\\
Email: xsun@math.ac.cn \\
\address{Mingshuo Zhou: School of Science, Hangzhou Dianzi University, Hangzhou 310018, P. R. of China.
\\ Institute of Mathematics, Academy of Mathematics and Systems Science, Chinese Academy of Sciences, Beijing 100190, P. R. of China.}\\
Email: zhoumingshuo@amss.ac.cn
\end{document}